\documentclass[10pt]{article} 
\usepackage{amssymb}
\usepackage{amsmath,amsfonts,amsthm}

\newtheorem{theorem}{Theorem}[section]
\newtheorem{lemma}[subsection]{Lemma}
\newtheorem{proposition}[theorem]{Proposition}
\newtheorem{corollary}[theorem]{Corollary}
\newtheorem{remark}[theorem]{Remark}
\newtheorem{definition}[theorem]{Definition}

\usepackage{fullpage}

\begin{document}

\numberwithin{equation}{section}

\title{A Minkowski inequality for the static Einstein-Maxwell space-time}

\author{Benedito Leandro$^1$, Ana Paula de Melo$^2$ and Hudson Pina$^3$}

\footnotetext[1]{Universidade Federal de Goi\'as, IME, CEP 74690-900, Goi\^ania, GO, Brazil. \textsf{bleandroneto@ufg.br}}

\footnotetext[2]{Universidade Federal de Goi\'as, IME, CEP 74690-900, Goi\^ania, GO, Brazil. \textsf{anapmelocosta@gmail.com}}
	 \footnotetext[2]{Ana Paula de Melo was partially supported by PROPG-CAPES [Finance Code 001].}

\footnotetext[3]{Universidade Federal de Mato Grosso, CEP 78600-000, Barra do Garças, MT, Brazil. \textsf{hudsonmat@gmail.com}}

\date{}

\maketitle{}

\begin{abstract}
In this paper we prove a Minkowski-like inequality for an asymptotically flat static Einstein-Maxwell (electrostatic) space-time using as approach the inverse mean curvature flow (IMCF). Moreover, we discuss the importance of this inequality in the understanding of the photon sphere.
\end{abstract}

\vspace{0.2cm} \noindent \emph{2020 Mathematics Subject
Classification} : 53C20; 53E10; 53C21.

\vspace{0.4cm}\noindent \emph{Keywords}: Einstein-Maxwell equation; Asymptotically flat; Minkowski inequality; Photon spheres.

\ 

\section{Introduction and Main Results}

The Minkowski inequality is an estimate from below for the total mean curvature of a hypersurface of a given Riemannian manifold (cf. \cite{brendle,huisken,li2017,mccormick,wei} and the references therein). Physically, this inequality is related with the Penrose inequality, an isoperimetric inequality (cf. \cite{gibbons,huisken}).

In this work we will prove the Minkowski inequality for the dimensionally reduced static Einstein–Maxwell (or electrostatic) manifold (cf. \cite{cederbaum2016,sophia,yaza2015} and the references therein):

\begin{definition} 
\label{defA} A Riemannian manifold $(M^{n},g)$ such that $f,\,\psi:M\rightarrow\mathbb{R}$ are smooth functions satisfying
\begin{eqnarray}\label{principaleq}
    f{Ric}={\nabla^2}f- \frac{2}{f}\nabla\psi\otimes\nabla\psi + \frac{2}{(n-1)f} |\nabla\psi|^{2}g,
\end{eqnarray}
\begin{eqnarray}\label{lapla}
\Delta f=2\left(\frac{n-2}{n-1}\right)\frac{|\nabla\psi|^{2}}{f}
\end{eqnarray} 
and
\begin{eqnarray}\label{diver}
div\left(\frac{\nabla\psi}{f}\right)=0,
\end{eqnarray}
where $Ric$ and $R$ denote the Ricci tensor and the scalar curvature of $(M^n,\, g).$ Moreover, ${\nabla}^{2}$ and $div$ are, respectively, the Hessian and the divergence for $g$ and $\Delta$ is the Laplacian operator. Furthermore, $f>0$ on $M$, $\partial M$ is the boundary of $M$ and $f=0$ in $\partial M$. Then $(M^n,\,g,\,f,\,\psi)$ is called an {\it electrostatic system.}
\end{definition} 
The above definition implies that the scalar curvature $R$ for the metric $g$ is given by
\begin{eqnarray}\label{scalarcurv}
f^{2}R=2|\nabla\psi|^{2}.
\end{eqnarray}

A classical solution for the electrostatic system is the Reissner–Nordstr\"om manifold. It was proved that an asymptotically flat static Einstein-Maxwell space (with some suitable initial boundary conditions) must be the Reissner–Nordstr\"om manifold (cf. \cite{chrusciel,sophia,Ruback} for a good over view). This solution represents a model for a static black hole electrically charged. The classification of static Einstein-Maxwell space for the general case, however,
still remains open.

Let $\Omega$ be a bounded domain with smooth boundary in $(M^n,\,g)$. Then there are two cases to consider:
\begin{itemize}
    \item $\Omega$ has only one boundary component $\Sigma=\partial\Omega$ and we say that $\Sigma$ is null-homologous;
    \item $\Omega$ has only one boundary component $\Sigma\cup\partial M=\partial\Omega$ and we say that $\Sigma$ is homologous to the horizon $\partial M$.
\end{itemize}
The boundary hypersurface $\Sigma$ is said to be outward minimizing if whenever $E$ is a domain containing $\Omega$, then $|\partial E|\geq|\partial \Omega|.$ An outward minimizing hypersurface must be a mean-convex hypersurface.

We will apply the inverse mean curvature flow (IMCF) to prove the Minkowski inequality, which we define now (cf. \cite{bethuel2006,huisken,huisken2008,huisken2009}). The classical (smooth) IMCF is a one-parameter family of hypersurfaces $\Sigma_t$ given by $x : \Sigma\times [0,\,T )\rightarrow M$ that evolves with speed proportional to the reciprocal of the mean curvature:
\begin{eqnarray}\label{IMCF}
\frac{\partial x}{\partial t}=\frac{1}{H}\nu,
\end{eqnarray}
where $\nu$ is the unit normal pointing towards infinity and $H(x,\,t)$ is the mean curvature of $\Sigma_t$. We assume the mean curvature $H(x,\,0)$ for the initial hypersurface $\Sigma$ is positive. In general, the flow does not remain smooth for all time (see the torus example in \cite{huisken}), and one must work with a weak formulation of IMCF, which properly jumps past times where the flow fails to be smooth. However, it was proved that if $H(x,\,t)$ is positive the IMCF is smooth (cf. $28'26''$ in \cite{huisken2009}). 
Therefore, in this paper we assume from now on that the IMCF will work accordingly, and for that we will accept $H(x,\,t)>0$ to avoid the technicalities which appear in the IMCF like in \cite{li2017,mccormick,wei}. 

In \cite{huisken}, Huisken and Ilmanen developed the weak solution of IMCF to overcome the smootheness problem of the IMCF. The evolving hypersurfaces are given by the level-sets of a scalar function $u:M\rightarrow\mathbb{R}$ such that $$\Omega_t=\{(x,\,t)\in\Sigma\times[0,\,\infty);\, u(x,\,t)\leq t\}$$
and 
 $$\Sigma_t=\{(x,\,t)\in\Sigma\times[0,\,\infty);\, u(x,\,t) = t\}.$$
Whenever $u$ is smooth with non vanishing gradient, $\nabla u\neq0$. Under our settings there will exist a weak solution for IMCF (cf. \cite{huisken}). In \cite{huisken}, Huisken and Ilmanen ensured the existence of a weak solution for the IMCF for an asymptotically flat manifold with dimension $n\leq7.$ Moreover, the weak solution will be smooth for all $t$ if the mean curvature of $\Sigma_t$ is positive. 

Now, we define on each $\Sigma_t$ the functional
\begin{eqnarray*}
Q(t):=|\Sigma_{t}|^{-\frac{n-2}{n-1}}\left(\int_{\Sigma_{t}}fHd\mu_{t}-n(n-1)\int_{\Omega_t} fR dv_t\right),
\end{eqnarray*}
where $|\Sigma_t|$ is the area of $\Sigma_t.$ We will prove that this is monotone along the weak inverse mean curvature flow (IMCF). This monotonicity has been
used previously (cf. \cite{brendle,li2017,mccormick,wei}) to provide the Minkowski inequalities in Schwarzschild, Kottler, Schwarzschild-AdS manifolds and asymptotically flat static vacuum space-times in general.

Our goal with this paper is to take a glance at how the Minkowski inequality proved by \cite{brendle,li2017,mccormick,wei} can be generalized to an asymptotically flat electrostatic system. The Schwarzschild, Reissner–Nordstr\"om and Majumdar-Papapetrou solutions are well-known examples of electrostatic systems.

 The scalar curvature $R$ for the Reissner–Nordstr\"om manifold is given by: $$R=\dfrac{(n-1)(n-2)q^{2}}{r^{2(n-1)}},$$
where $q$ is the electric charge and $r$ is the radial coordinate of the black hole of ADM mass $m$. Moreover, the Reissner–Nordstr\"om manifold is well defined for $r>(m+\sqrt{m^{2}-q^{2}})^{1/(n-2)}$. So, we can see that the scalar curvature of the Reissner-Nordstr\"om solution is decreasing as long $r$ increases, and goes to zero at infinity. Moreover, is reasonable to expect that the scalar curvature will
be bounded from above. A straightforward computation shows us that the scalar curvature for the Reissner-Nordstr\"om space is {\it subharmonic} (see \cite[Remark 5]{sophia}), i.e., $$\Delta R\geq0.$$

The idea of the proof is simply to mix IMCF with the flow generated by the normal exponential map for a conformal metric.

Without further ado, we state our main result.

\noindent{\bf Main Theorem.}{\it\,
 Consider that $(M^n,\,g,\,f,\,\psi)$, $3\leq n\leq7$, is an asymptotically flat electrostatic manifold such that the scalar curvature is subharmonic. Let $\Omega\subset M$ denote the region bounded by $\Sigma$, null-homologous, that is outer-minimizing with positive mean curvature $H$. Then, we have
\begin{eqnarray*}
\int_{\Sigma}fHdS-n(n-1)\int_{\Omega} fR dv\geq(n-1)(w_{n-1})^{\frac{1}{(n-1)}}|\Sigma|^{\frac{n-2}{n-1}},
\end{eqnarray*}
where $w_{n-1}$ is the area of the $(n-1)$-dimensional standard sphere of radius 1.}

\begin{remark}
Our result is a natural generalization of the classical Minkowski inequality for convex hypersurface $\Sigma$ in $\mathbb{R}^n$, which states that
\begin{eqnarray*}
\int_{\Sigma}HdS\geq(n-1)(w_{n-1})^{\frac{1}{(n-1)}}|\Sigma|^{\frac{n-2}{n-1}}.
\end{eqnarray*}
\end{remark}
\begin{remark}
The homologous case for the electrostatic system is more delicate. In \cite{brendle2012}, the author showed the differences between the null-homologous and the homologous cases for the Reissner-Nordstr\"om space. Also, in \cite[Theorem 1.1]{wei} this distinction was carefully considered. 
\end{remark}

To prove our next result we recommend to the reader to see Remark \ref{asymphomo}. 
\begin{corollary}
 Consider that $(M^n,\,g,\,f,\,\psi)$, $3\leq n\leq7$, is an asymptotically flat electrostatic manifold such that the scalar curvature is subharmonic. Let $\Omega\subseteq M$ with (homologous) boundary $\partial\Omega=\Sigma\cup\partial M$ such that $\Sigma$ is outer-minimizing with positive mean curvature $H$. Then, we have
\begin{eqnarray*}
\int_{\Sigma}fHdS-n(n-1)\int_{\Omega} fR dv\geq(n-1)(w_{n-1})^{\frac{1}{(n-1)}}|\Sigma|^{\frac{n-2}{n-1}} - 2m(n-1)w_{n-1},
\end{eqnarray*}
where $w_{n-1}$ is the area of the $(n-1)$-dimensional standard sphere of radius 1.
\end{corollary}

The Reissner–Nordstr\"om photon spheres model photons spiraling around the central black hole or naked singularity ‘at a fixed distance’. A Reissner–Nordstr\"om space-time can contain a photon sphere (or not), on which light can get trapped (cf. \cite{cederbaum2016,galloway,sophia,yaza2015} and the references therein). The existence of such photon sphere depends on the relationship between mass and electric charge of the black hole. 


It was proved that a photon sphere in an asymptotically flat electrostatic system with an electro-magnetic energy–momentum tensor satisfying the NEC (null energy condition) has positive mean curvature (cf. Theorem 3.1 in \cite{galloway} and Lemma 2.6 in \cite{cederbaum2017}). Thus, the positivity of the mean curvature of $\Sigma$ in the Main Theorem can be replaced by the NEC (when $\Sigma$ is a photon sphere).
This shows that the technicality of the IMCF can be avoided assuming some particular (and natural) hypothesis, although as far as we know there are
no known examples of null-homologous photon spheres in an electrostatic
system. We remember that in the photon sphere $\Sigma$ the mean curvature  and the static potential $f$ are constants.

\section{Background}

 In what follows we will define what is an asymptotically flat space-time (cf. \cite{kunduri2018,mccormick,Ruback,yaza2015}). From now on, the scalar curvature is subharmonic.

\begin{definition}\label{def2}
	A solution $(M^{n},\,g,\,f,\,\psi)$ for \eqref{principaleq}, \eqref{lapla} and \eqref{diver} is said to be asymptotically flat with one end if $M$ minus a compact set is diffeomorphic to $\mathbb{R}^{n}$ minus a closed ball centered at
the origin, and the metric $g$ and the static potential $f$ satisfy the following asymptotic
	expansions at infinity.
	\begin{itemize}
		\item[(I)] Let $r^2=\|x\|$, $x\in M$, $\delta$ be the flat metric and $\eta_{ij}(x)=o(r^{2-n})$ as $r\rightarrow\infty$,  $$g_{ij}(x)=\delta_{ij}(x)+\eta_{ij}(x).$$ 
		\item[(II)] For $\omega=o(r^{2-n})$, $r\rightarrow\infty$,
		\begin{eqnarray*}
			f=1-\dfrac{m}{r^{n-2}}+\omega.
		\end{eqnarray*}
		\item[(IV)] Moreover,
		\begin{eqnarray*}
			\partial_{l}\eta_{ij}=o(r^{1-n});\quad \partial_{i}\omega=o(r^{-n})\quad\mbox{and}\quad\partial_{i}\partial_{j}\omega=o(r^{-(n+1)}),
		\end{eqnarray*}
	\end{itemize}
	where $1\leq l,\,i,\,j\leq n$. Here $m\in\mathbb{R}$ is the ADM mass.
\end{definition}

For an asymptotically flat  space-time a straightforward computation ensures us that 
\begin{eqnarray}\label{intinfty}
\displaystyle\lim_{r\rightarrow\infty}\int_{\mathbb{S}(r)}\langle\nabla f,\,\eta\rangle dS
	=(n-2)m\displaystyle\lim_{r\rightarrow\infty}\dfrac{1}{r^{n-1}}\int_{\mathbb{S}(r)}dS=(n-2)w_{n-1}{m},
\end{eqnarray}
where $\eta$ is the outward normal vector field of the $(n-1)$-dimensional standard sphere $\mathbb{S}$ of radius $r$, and $w_{n-1}$ is the area of the $(n-1)$-dimensional standard sphere $\mathbb{S}$ of radius 1.

The proof of the Main Theorem is based on \cite{brendle2012,brendle,mccormick,wei}. That being said, we remember some preliminar facts used by Brendle \cite{brendle2012}. 


We assume here that $\partial\Omega=\Sigma$ (null-homologous). Let $\nu$ denote the outward-pointing unit normal to $\Sigma$. We will accept throughout this section that $\Sigma$ has positive mean curvature with respect to this choice of unit normal. Consider the conformally modified metric $\hat{g}=\dfrac{1}{f^{2}}g$ (cf. \cite{khunel1988} for the conformal changes of the metric). 
 
 Furthermore, denote by $\Phi:\Sigma\times[0,\,\infty)\rightarrow\hat{\Omega}$ the normal exponential map (cf. \cite{petersen}) with respect to the metric $\hat{g}$. More precisely, for each point $x\in\Sigma$ the curve $\gamma_{x}(t)=\Phi(x,\,t)$ is a geodesic with respect to $\hat{g}$, and then
\begin{eqnarray*}
\Phi(x,\,0)=x\quad\mbox{and}\quad \Phi(x,\,t)=\exp_{x}{\big(-tf(\gamma(t))\nu\big)}.
\end{eqnarray*}
Therefore,
\begin{eqnarray}\label{confflow}
\frac{\partial\Phi}{\partial t}(x,\,t)\Bigg|_{t=0}=-f(x)\nu(x).
\end{eqnarray}
Note that the geodesic $\gamma$ has unit speed with respect to $\hat{g}.$ 

Next we define the following sets
\begin{eqnarray}
\Sigma_t=\{(x,\,t)\in\Sigma\times[0,\,\infty):\, u(\Phi(x,\,t))=t\}
\end{eqnarray}
and
\begin{eqnarray}
\Omega_t=\{(x,\,t)\in\Sigma\times[0,\,\infty):\,(x,\,t+\delta)\in \Sigma_t;\,\delta>0\},
\end{eqnarray}
where $u=dist_{\hat{g}}(p,\,\Sigma)$ is the distance function of $p\in\hat{\Omega}$ from $\Sigma$ with respect to $\hat{g}$.
The set $\Sigma_t$ is closed, and we have $\Phi(\Sigma_t)=\hat{\Omega}$ (cf. Proposition 3.1. in \cite{brendle2012}). To fix notation, we denote by $H$ and $h$ the mean curvature and second fundamental form of $\Sigma_t$ with respect to the metric $g.$

The following lemmas will be true for a solution $(M^{n},\,g,\,f,\,\psi)$ of the electrostatic system.

\begin{lemma}\label{smootheness}
Assume that $\partial\Omega=\Sigma$ has positive mean curvature. Then, the hypersurface $\Sigma_t$ has positive mean curvature for each $t$ and satisfies the  differential inequality $$\frac{\partial}{\partial t}\left(\frac{fR}{H}\right)\leq\frac{-Rf^{2}}{(n-1)}.$$

\end{lemma}
\begin{proof}
Since 
\begin{eqnarray*}
\Delta^{\Sigma_t}f+\nabla^{2}f(\nu,\,\nu)=\Delta f-H\langle\nabla f,\,\nu\rangle.
\end{eqnarray*}
From Definition \ref{defA} we get
\begin{eqnarray*}
\Delta^{\Sigma_t}f+f{Ric}(\nu,\,\nu) +\frac{2}{f}\left(\langle\nabla\psi,\,\nu\rangle^{2}-\frac{|\nabla\psi|^{2}}{n-1}\right)= 2\left(\frac{n-2}{n-1}\right)\frac{|\nabla\psi|^{2}}{f} -H\langle\nabla f,\,\nu\rangle.
\end{eqnarray*}
That is,
\begin{eqnarray*}
\Delta^{\Sigma_t}f+f{Ric}(\nu,\,\nu) = \frac{-2}{f}\left(\langle\nabla\psi,\,\nu\rangle^{2}-|\nabla\psi|^{2}\right) -H\langle\nabla f,\,\nu\rangle.
\end{eqnarray*}
Then, from Cauchy-Schwarz inequality we have
\begin{eqnarray}\label{AMP}
\Delta^{\Sigma_t}f\geq-fRic(\nu,\,\nu)-H\langle\nabla f,\,\nu\rangle.
\end{eqnarray}

Assuming that $\frac{\partial}{\partial t}\Phi=-f(\Phi(x,\,t))\nu,$ where $\nu=-\dfrac{\nabla u}{|\nabla u|}$ denotes the outward-pointing unit normal vector to $\Sigma_t$ with respect to the metric $g$. Hence, the mean curvature of $\Sigma_t$ satisfy the evolution equations (cf. \cite{bethuel2006,brendle,huisken,huisken2008}):
\begin{eqnarray*}
\frac{\partial H}{\partial t}=\Delta^{\Sigma_t}f+f(|h|^{2}+Ric(\nu,\,\nu)).
\end{eqnarray*}
Therefore, from \eqref{AMP} we get
\begin{eqnarray*}
\frac{\partial H}{\partial t}\geq-H\langle\nabla f,\,\nu\rangle+f|h|^{2}.
\end{eqnarray*}

Thus, since $$    \frac{\partial f}{\partial t}=d(f\circ\Phi)(\partial_t)=\langle\nabla f,\,\frac{\partial\Phi}{\partial t}\rangle=-f\langle \nabla f,\nu \rangle\quad\mbox{and}\quad (n - 1)|h|^2 \geq H^2$$ we get
\begin{eqnarray*}
\frac{\partial}{\partial t}\left(\frac{fR}{H}\right)=\frac{1}{H}\frac{\partial}{\partial t}(fR)-\frac{fR}{H^{2}}\frac{\partial}{\partial t}H\leq\frac{f}{H}\frac{\partial R}{\partial t}-\frac{Rf^{2}}{(n-1)}. 
\end{eqnarray*}
Considering that $\Delta R\geq0$ and $R\geq0$, from Hopf's lemma we get
\begin{eqnarray*}
\frac{\partial}{\partial t}\left(\frac{fR}{H}\right)\leq-\frac{f^2}{H}\langle\nabla R,\,\nu\rangle-\frac{Rf^{2}}{(n-1)}\leq\frac{-Rf^{2}}{(n-1)}. 
\end{eqnarray*}

On the other hand, making a similar computation we get
\begin{eqnarray*}
\frac{\partial}{\partial t}\left(\frac{H}{f}\right)=\frac{1}{f}\frac{\partial}{\partial t}H-\frac{H}{f^{2}}\frac{\partial}{\partial t}f\geq\frac{H^{2}}{(n-1)}
\end{eqnarray*}
at each point on $\Sigma_t.$

By integration, we obtain
\begin{eqnarray*}
\frac{H}{f}(x,\,\tau)-\frac{H}{f}(x,\,0)\geq\int_{0}^{\tau}\frac{H^{2}}{(n-1)}dt,
\end{eqnarray*}
where $\tau\in[0,\,\infty).$

Since the initial hypersurface $\Sigma$ has positive mean curvature and $f>0$ in $M$, we conclude that the hypersurface $\Sigma_t$ has positive mean curvature for each $t$. Moreover, the area form on $\Sigma_t$ is monotone decreasing in $t$ (cf. \cite[Lemma 7.4]{bethuel2006}). 
\end{proof}

Now,  for a bounded static potential $f$ we define on each $\Sigma_t$ the functional
\begin{eqnarray*}
\tilde{Q}(t):=(n-1)\int_{\Sigma_{t}}\frac{fR}{H}d\mu_{t}
\end{eqnarray*}

\begin{lemma}\label{qrt}
Assume that $\partial\Omega=\Sigma$ with positive mean curvature. Then,
\begin{eqnarray*}
(n-1)\int_{\Sigma}\frac{fR}{H}d\mu\geq n\int_{\Omega}fRdv.
\end{eqnarray*}
\end{lemma}
\begin{proof}
From \eqref{confflow}, using that $\dfrac{\partial}{\partial t}(d\mu_t)=-fHd\mu_t$ (cf. \cite[Lemma 7.4]{bethuel2006}) we get
\begin{eqnarray*}
\frac{\partial}{\partial t}\tilde{Q}(t)
&=&(n-1)\int_{\Sigma_t}\dfrac{\partial}{\partial t}\left(\frac{fR}{H}\right)d\mu_t + (n-1)\int_{\Sigma_t}\left(\frac{fR}{H}\right)\dfrac{\partial}{\partial t}(d\mu_t)\nonumber\\
&=&(n-1)\int_{\Sigma_t}\dfrac{\partial}{\partial t}\left(\frac{fR}{H}\right)d\mu_t - (n-1)\int_{\Sigma_t}\left(\frac{fR}{H}\right)fHd\mu_t.
\end{eqnarray*}

Now, we use Lemma \ref{smootheness} to obtain
\begin{eqnarray*}
\frac{\partial}{\partial t}\tilde{Q}(t)
&=&(n-1)\int_{\Sigma_t}\dfrac{\partial}{\partial t}\left(\frac{fR}{H}\right)d\mu_t - (n-1)\int_{\Sigma_t}Rf^{2}d\mu_t\nonumber\\
&\leq& -n\int_{\Sigma_t}Rf^{2}d\mu_t.
\end{eqnarray*}
Then, by integrating the above inequality from $0$ to $\tau$, where $\tau\in[0,\,\infty)$, and using that the volume element is $dv$ from co-area formula we get
\begin{eqnarray*}
\tilde{Q}(0)-\tilde{Q}(\tau)\geq n\int_{0}^{\tau}\left(\int_{\Sigma_t}Rf^{2}d\mu_t\right)dt=n\int_{\Omega_\tau}fRdv,
\end{eqnarray*}
where we use that $u=dist_{\hat{g}}(p,\,\Sigma)$ is the distance function with respect to $\hat{g}$. In fact, $1=|\hat{\nabla} u|_{\hat{g}}=f|\nabla u|$.

Consequently, since $\tilde{Q}(\tau)>0$ we have
\begin{eqnarray*}
\tilde{Q}(0)\geq n\int_{\Omega_\tau}fRdv.
\end{eqnarray*}
Then, from Lemma \ref{smootheness}, for $\tau\rightarrow\infty$,
\begin{eqnarray*}
(n-1)\int_{\Sigma}\frac{fR}{H}d\mu\geq n\int_{\Omega}fRdv.
\end{eqnarray*}
\end{proof}

\begin{remark}\label{asymphomo}
For the homologous case, i.e., $\partial\Omega=\Sigma\cup\partial M,$ the above lemma still holds if we consider a decay at infinity for $R$. Let us consider that $$R=o(r^{-2(n-1)}),$$
for $r\rightarrow\infty.$ Hence, we have
\begin{eqnarray*}
\int_{\Sigma}Rf^{2}d\mu = \int_{\partial\Omega}Rf^{2}d\mu - \int_{\partial M}Rf^{2}d\mu.
\end{eqnarray*}
On the other hand, by Definition \ref{defA}, consider $\partial M=f^{-1}(0)\cup E_{\infty}$, where $E_{\infty}$ is the asymptotically flat end of the manifold. Here, we consider the most general case for $\Omega$. It is possible to have the case where $\partial \Omega= f^{-1}(0)\cup\Sigma.$ Therefore, assuming an asymptotic condition over $R$ we obtain
\begin{eqnarray*}
\int_{\partial M}Rf^{2}d\mu = \displaystyle\lim_{r\rightarrow\infty}\int_{\mathbb{S}^{n-1}}Rf^{2} dS = 0.
\end{eqnarray*}
Thus, 
\begin{eqnarray*}
\int_{\Sigma}Rf^{2}d\mu = \int_{\partial\Omega}Rf^{2}d\mu
\end{eqnarray*}
and Lemma \ref{qrt} holds.
\end{remark}

\section{Proof of the main result}

 Remember that the evolution equations under IMCF are well-known, to be given by equations (1.1) and (1.2) in \cite{huisken} (see also \cite{bethuel2006,huisken2008}):
\begin{eqnarray}\label{eq1}
\frac{\partial H}{\partial t}=-\Delta^{\Sigma_t}H^{-1}-H^{-1}(|h|^{2}+Ric(\nu,\,\nu))
\end{eqnarray}
and
\begin{eqnarray}\label{eq2}
\frac{\partial d\mu_t}{\partial t}=d\mu_t.
\end{eqnarray}

The following proposition is a fundamental key to prove the Minkowski inequality (cf. \cite{brendle,mccormick,wei}).
\begin{proposition}\label{proptotal}
Let $\Sigma_{t}$ (null-homologous) be a weak solution to IMCF (outer-minimizing) for $0 < t_1 < t_2 < T$ such that $\Sigma$ has positive mean curvature on a static Einstein-Maxwell manifold $M$. Then $$Q(t_2)\leq Q(t_1).$$
\end{proposition}
\begin{proof}
A straightforward computation from \eqref{eq2} gives us
\begin{eqnarray*}
    \frac{\partial}{\partial t}\int_{\Sigma_t}fHd\mu_t=\int_{\Sigma_t}\left(H \frac{\partial f}{\partial t}+f\frac{\partial H}{\partial t}+fH\right)d\mu_t.
\end{eqnarray*}
Then, from \eqref{IMCF} and \eqref{eq1} we obtain:
\begin{eqnarray*}
    \frac{\partial}{\partial t}\int_{\Sigma_t}fHd\mu_t&=&\int_{\Sigma_t}\left(H\langle \nabla f,\,\nu\rangle\frac{1}{H}-f\Delta^{\Sigma_t}H^{-1}-fH^{-1}(|h|^{2}+Ric(\nu,\,\nu))+fH\right)d\mu_t\nonumber\\
    &\leq& \int_{\Sigma_t}\left(\langle \nabla f,\,\nu\rangle-H^{-1}(\Delta^{\Sigma_t}f+fRic(\nu,\,\nu))+\frac{n-2}{n-1}fH\right)d\mu_t,
\end{eqnarray*}
where we have used the inequality $(n - 1)|h|^2 \geq H^2$ and the fact that $\Sigma_t$ is closed. The equality holds if $\Sigma_t$ is totally umbilic.

Since 
\begin{eqnarray*}
\Delta^{\Sigma_t}f+\nabla^{2}f(\nu,\,\nu)=\Delta f-H\langle\nabla f,\,\nu\rangle
\end{eqnarray*}
from \eqref{principaleq} and \eqref{lapla} we get
\begin{eqnarray*}
\Delta^{\Sigma_t}f+f{Ric}(\nu,\,\nu) +\frac{2}{f}\left(\langle\nabla\psi,\,\nu\rangle^{2}-\frac{|\nabla\psi|^{2}}{n-1}\right)= 2\left(\frac{n-2}{n-1}\right)\frac{|\nabla\psi|^{2}}{f} -H\langle\nabla f,\,\nu\rangle.
\end{eqnarray*}
That is,
\begin{eqnarray}\label{correcao}
\Delta^{\Sigma_t}f+f{Ric}(\nu,\,\nu) = \frac{-2}{f}\left(\langle\nabla\psi,\,\nu\rangle^{2}-|\nabla\psi|^{2}\right) -H\langle\nabla f,\,\nu\rangle.
\end{eqnarray}

Therefore, by Cauchy-Schwarz inequality and assuming that $H>0$, from \eqref{correcao} we have
\begin{eqnarray*}
    \frac{\partial}{\partial t}\int_{\Sigma_t}fHd\mu_t\leq \int_{\Sigma_t}\left[2\langle \nabla f,\,\nu\rangle+\frac{n-2}{n-1}fH+\frac{2}{fH}\left(\langle\nabla\psi,\,\nu\rangle^{2}-|\nabla\psi|^{2}\right)\right]d\mu_t\nonumber\\
    \leq \int_{\Sigma_t}\left[2\langle \nabla f,\,\nu\rangle+\frac{n-2}{n-1}fH\right]d\mu_t.
\end{eqnarray*}

If $\Sigma_t=\partial\Omega_t$ is null-homologous for all $t$ we have
\begin{eqnarray*}
 \int_{\Sigma_t}\langle \nabla f,\,\nu\rangle d\mu_t= \int_{\Omega_t}\Delta f dv_t 
\end{eqnarray*}
Then, from \eqref{lapla} and \eqref{scalarcurv} we get $$\Delta f =\frac{n-2}{n-1}Rf.$$ Moreover, we obtain
\begin{eqnarray*}
 \int_{\Sigma_t}\langle \nabla f,\,\nu\rangle d\mu_t = \frac{(n-2)}{(n-1)}\int_{\Omega_t} fR dv_t.
\end{eqnarray*}

Since $\Sigma_t$ has positive mean curvature, we now can invoke Lemma \ref{qrt}, i.e.,
\begin{eqnarray*}
n\int_{\Omega_t} fR dv_t\leq (n-1)\int_{\Sigma_t} \frac{fR}{H} d\mu_t.
\end{eqnarray*}
Then, using $$\frac{\partial}{\partial t}\int_{\Omega_t}fR dv_t=\int_{\Sigma_t}\frac{fR}{H} d\mu_t$$ (cf. Proposition 9.1. in \cite{scheuer}) we have
\begin{eqnarray*}
    &&\frac{\partial}{\partial t}\left(\int_{\Sigma_t}fHd\mu_t-n(n-1)\int_{\Omega_t} fR dv_t\right)\leq \int_{\Sigma_t}\left[2\langle \nabla f,\,\nu\rangle+\frac{n-2}{n-1}fH\right]d\mu_t - n(n-1)\int_{\Sigma_t} \frac{fR}{H} d\mu_t \nonumber\\
    &&\leq\frac{n-2}{n-1}\int_{\Sigma_t}fHd\mu_t + 2\frac{(n-2)}{(n-1)}\int_{\Omega_t} fR dv_t -n(n-1)\int_{\Sigma_t} \frac{fR}{H} d\mu_t\nonumber\\
    &&\leq\frac{n-2}{n-1}\int_{\Sigma_t}fHd\mu_t   +\left[\dfrac{2(n-2)}{(n-1)}-n^2\right]\int_{\Omega_t} fR dv_t \nonumber\\
       &&=\frac{n-2}{n-1}\left(\int_{\Sigma_t}fHd\mu_t  +\left[2-\dfrac{n^2(n-1)}{(n-2)}\right]\int_{\Omega_t} fR dv_t\right)\nonumber\\ &&\leq\frac{n-2}{n-1}\left(\int_{\Sigma_t}fHd\mu_t-n(n-1)\int_{\Omega_t} fR dv_t\right).
\end{eqnarray*}

Considering $P(t)=\int_{\Sigma_t}fHd\mu_t  - n(n-1)\int_{\Omega_t} fR dv_t$ we can conclude that
\begin{eqnarray*}
\frac{\partial}{\partial t}P(t)\leq \left(\frac{n-2}{n-1}\right)P(t).
\end{eqnarray*}
Hence, by integration we get
\begin{eqnarray}\label{danada}
P(t_2)\leq P(t_1) e^{\frac{n-2}{n-1}(t_{2}-t_{1})},
\end{eqnarray}
for all $0\leq t_1\leq t_2.$ Since $\Sigma_t$ is outward minimizing, we have that $|\Sigma_t|=e^{t}|\Sigma|$ for all $t$. Therefore, from a simple computation we have $$|\Sigma_{t_{2}}|=e^{t_{2}-t_{1}}|\Sigma_{t_{1}}|.$$
Plugging this into \eqref{danada} we have $$Q(t_{2})\leq Q(t_{1}).$$
We recommended to the reader section 4.2 in \cite{wei} to conclude the monotonicity of the functional $Q(t).$
Here, the smootheness of the IMCF is guaranteed by Lemma \ref{smootheness}.
\end{proof}

\begin{remark}
In the homologous case, following \cite[Theorem 3.2]{wei} we define the functional 
\begin{eqnarray*}
Q(t):=|\Sigma_{t}|^{-\frac{n-2}{n-1}}\left(2m(n-1)w_{n-1}+ \int_{\Sigma_{t}}fHd\mu_{t}-n(n-1)\int_{\Omega_t} fR dv_t\right).
\end{eqnarray*}
Moreover, considering the asymptotic conditions from \eqref{lapla} and \eqref{scalarcurv} we get
\begin{eqnarray}\label{laplahomo}
\int_{\Sigma}\langle \nabla f,\,\nu\rangle d\mu&=&m(n-2)w_{n-1} - \int_{\partial M}\langle \nabla f,\,\nu\rangle d\mu + \frac{n-2}{n-1}\int_{\Omega}Rf dv\nonumber\\
&\leq& m(n-2)w_{n-1} + \frac{n-2}{n-1}\int_{\Omega}Rf dv,
\end{eqnarray}
where we use Hopf's lemma in the last inequality. Remember that $\Delta f\geq0$ and $f\geq0$. Therefore, following the same steps of Proposition \ref{proptotal} we can derive the monotonicity of $Q(t)$.
\end{remark}

 Under the same conditions of the proposition above we can assume for any asymptotically flat manifold that the following lemma is true and the proof can be found in \cite{wei} (see also \cite{huisken}). This fact was observed by Mccormick in \cite{mccormick}.

\begin{lemma}\label{lemmaWei}
 Let $(M^n,\,g)$ be an asymptotically flat manifold. Consider that $\Sigma_{t}$ is a smooth weak solution to IMCF. Then, we have
\begin{eqnarray*}
\lim_{t\rightarrow\infty}|\Sigma_{t}|^{-\frac{n-2}{n-1}}\left(\int_{\Sigma_{t}}fHd\mu_{t}\right)=(n-1)(w_{n-1})^{1/(n-1)}.
\end{eqnarray*}
\end{lemma}

Assuming the above lemma works, we obtain the next result.

\begin{lemma}\label{lemma3.2}
 Let $(M^n,\,g)$ be an asymptotically flat manifold. Consider that $\Sigma_{t}=\partial\Omega_t$ is a smooth weak solution for the IMCF. Then, we have
$$\lim_{t\rightarrow\infty}Q(t)\geq(n-1)(w_{n-1})^{1/(n-1)}.$$
\end{lemma}
\begin{proof}
 Let $r(t)$ be such that $|\Sigma_t|=w_{n-1}r(t)^{n-1}$. 
 First of all, observe that $|\Sigma_t|\rightarrow\infty $ as $t\rightarrow\infty.$ Now, take the limit of $Q(t)$:
$$\lim_{t\rightarrow\infty}Q(t)=\lim_{t\rightarrow\infty}|\Sigma_{t}|^{-\frac{n-2}{n-1}}\left(\int_{\Sigma_{t}}fHd\mu_{t}-n(n-1)\int_{\Omega_t} fR dv_t\right).$$
Moreover, by integration we get
\begin{eqnarray*}
\int_{\Sigma_t}|\nabla f|\geq\int_{\Sigma_t}\langle\nabla f,\,\nu\rangle=\int_{\Omega_t}\Delta f=\dfrac{n-2}{n-1}\int_{\Omega_t}Rf.
\end{eqnarray*}

Thus, from Lemma \ref{lemmaWei} and the asymptotic conditions we get
\begin{eqnarray*}
&&\lim_{t\rightarrow\infty}Q(t)\geq (n-1)(w_{n-1})^{1/(n-1)}-n\frac{(n-1)^2}{(n-2)}\lim_{t\rightarrow\infty}|\Sigma_{t}|^{-\frac{n-2}{n-1}}\int_{\Sigma_{t}}|\nabla f|d\mu_t\nonumber\\
&&=(n-1)(w_{n-1})^{1/(n-1)}-nm(n-1)^2(w_{n-1})^{1/(n-1)}\lim_{t\rightarrow\infty}r(t)^{-(n-2)} = (n-1)(w_{n-1})^{1/(n-1)},
\end{eqnarray*}
where we assume that $r\rightarrow\infty$ as $t\rightarrow\infty$ and $|\nabla f|=(n-2)\frac{m}{r^{n-1}}$ at infinity (cf. Definition \ref{def2}).


\end{proof}

\begin{remark}
Also is easy to see that the above lemma holds for the homologous case. We just need to consider the first equality in \eqref{laplahomo}.
\end{remark}

 \noindent {\bf Proof of the Main Theorem:}
 Since $Q(t)$ is monotone decreasing, we conclude that $$Q(0)\geq\lim_{t\rightarrow\infty}Q(t)\geq(n-1)(w_{n-1})^{1/(n-1)}.$$ From this, the theorem follows immediately.
 
 \hfill$\Box$



\begin{thebibliography}{99}

\bibitem{bethuel2006}{Bethuel, F., et al.} - {\em Calculus of Variations and Geometric Evolution Problems.} Lectures given at the 2nd Session of the Centro Internazionale Matematico Estivo (CIME) held in Cetaro, Italy, June 15-22, 1996. Springer, (2006). MR1730218

\bibitem{brendle2012}{Brendle, S.} - {\em Constant mean curvature surfaces in warped product manifolds.} Publ. Math. Inst. Hautes Études Sci. 117, (2013): 247-269. MR3090261

\bibitem{brendle}{Brendle, S., Hung, P-K. \& Wang, M-T.} - {\em A Minkowski inequality for hypersurfaces in the anti–de Sitter–Schwarzschild manifold.} Comm. Pure Appl. Math. 69(1), (2016): 124-144. MR3433631

\bibitem{cederbaum2016}{Cederbaum, C. \& Galloway, G.} - {\em Uniqueness of photon spheres in electro-vacuum spacetimes.} Class. Quantum Grav. 33(7), (2016): pp. 075006. MR3471730

\bibitem{cederbaum2017}{Cederbaum, C. \& Galloway, Gregory J.} - {\em Uniqueness of photon spheres via positive mass rigidity.} Comm. Anal. Geom. 25(2), (2017): 303-320. MR3690243

\bibitem{chrusciel}{Chru\'sciel, P. T.} - {\em Towards a classification of static electrovacuum spacetimes containing an asymptotically flat spacelike hypersurface with compact interior.} Class. Quantum Grav. 16 (1999), no. 3, 689–704. MR1682570 (2000c:83033)

\bibitem{galloway}{Galloway, G \& Miao, P.} - {\em Variational and rigidity properties of static potentials.} Comm. Anal. Geom. 25(1), (2017): 163-183. MR3663315

\bibitem{gibbons}{Gibbons, G. W.} - {\em Collapsing shells and the isoperimetric inequality for black holes.} Class. Quantum Grav. 14.10 (1997): 2905. MR1476553

  \bibitem{huisken} {Huisken, G. \& Ilmanen, T.} - {\em The inverse mean curvature flow and the Riemannian Penrose inequality.} J. Diff. Geom. 59, (2001): 353-437. MR1916951
  
  \bibitem{huisken2008}{Huisken, G. \& Ilmanen, T.} - {\em Higher regularity of the inverse mean curvature flow.} J. diff. geom. 80(3), (2008): 433-451. MR2472479
  
    \bibitem{huisken2009} {Huisken, G.} - {\em Inverse mean curvature flow and isoperimetric inequalities.} Video available at https://video.ias.edu/node/233 (2009).
  
  
  \bibitem{sophia}{Jahns, S.} - {\em Photon sphere uniqueness in higher-dimensional electrovacuum spacetimes.} Class. Quantum Grav. 36(23), (2019): 235019.
  
 \bibitem{khunel1988} {Khunel, W.} - {\em Conformal transformations between Einstein spaces.} - In: Conformal Geometry, ed. by R. S. Kulkarni and U. Pinkall, Aspects Math. E, vol. 12, Vieweg, Braunschweig, (1988): 105-146. MR0979791

\bibitem{kunduri2018} {Kunduri, H. K.; Lucietti, J.}  {\em No static bubbling spacetimes in higher dimensional einstein-maxwell theory.} Classical and Quantum Gravity 35(5) (2018), p.054003 (9pp). 
  
  \bibitem{li2017}{Li, H. \& Wei, Y.} - {\em On inverse mean curvature flow in Schwarzschild space and Kottler space.} Calc. Var. 56(62), (2017). MR3639616
  
  \bibitem{mccormick}{McCormick, S.} - {\em On a Minkowski-like inequality for asymptotically flat static manifolds.} Proc. AMS. 146(9) (2018): 4039-4046. MR3825857


 \bibitem{petersen}{Petersen, P.} - {\em  Riemannian Geometry.} Third edition. Graduate Texts in Mathematics, 171. Springer, Cham, 2016. xviii+499 pp. ISBN: 978-3-319-26652-7; 978-3-319-26654-1. MR3469435
 
 \bibitem{Ruback}{Ruback, P.} - {\em A new uniqueness theorem for charged black holes.} Class. Quantum Grav. 5 (1988), no. 10, L155–L159. MR0964972

\bibitem{scheuer}{Scheuer, J. \& Xia, C.} - {\em Locally constrained inverse curvature flows.} Transactions AMS. 372(10), (2019): 6771-6803. MR4024538

\bibitem{yaza2015}{Yazadjiev, S. \& Lazov, B.} - {\em Uniqueness of the static Einstein–Maxwell spacetimes with a photon sphere.} Class. and Quantum Grav. 32(16), (2015): 165021. MR3382637

\bibitem{wei} {Wei, Y.} - {\em On the Minkowski-type inequality for outward minimizing hypersurfaces in Schwarzschild space.} Calc. Var. 57(46), (2018). MR3772872




\end{thebibliography}
\end{document}